\documentclass[a4paper, 11pt]{article}
\usepackage[utf8]{inputenc}
\usepackage[T1]{fontenc}
\usepackage[english]{babel}
\usepackage{graphicx}
\usepackage{stmaryrd}
\usepackage{tikz}
\usepackage{tikz-cd} 
\usepackage[all]{xy}
\usepackage{comment}
\usepackage{bm}
\usepackage{comment}
\usepackage{amsmath,amsfonts,amssymb}
\usepackage[a4paper]{geometry}
\usepackage{amsthm}
\usepackage{float}
\usepackage{enumitem}
\usepackage{hyperref}
\usepackage{xcolor}
\hypersetup{
    colorlinks=true,
    linkcolor={blue},
    citecolor={blue},
    urlcolor={blue}}
\newtheorem{te}{Theorem}[section]

\newtheorem{qu}[te]{Question}
\newtheorem{lemme}[te]{Lemma}

\theoremstyle{definition}

\theoremstyle{remark}
\newtheorem{rk}[te]{Remark}

\newlength{\plarg}
\usepackage{IEEEtrantools}

\usepackage{typearea}
\usetikzlibrary{cd}
\newcommand{\norm}[1]{\lvert\lvert#1\rvert\rvert}

\usepackage[backend=biber,sorting=nyt, style=alphabetic]{biblatex}
\title{Non-split sharply 2- and 3-transitive groups in $\mathrm{SL}_n(\mathbb{Z})$}
\author{Marco Amelio and Simon André}
\begin{document}

\begin{minipage}{\linewidth}

\maketitle

\begin{abstract}
\noindent We prove that $\mathrm{SL}_3(\mathbb{Z})$ contains a non-split sharply 2-transitive subgroup, answering a question of Glasner and Gulko. We also prove that $\mathrm{SL}_4(\mathbb{Z})$ contains a non-split sharply 3-transitive subgroup, but that $\mathrm{SL}_3(\mathbb{Z})$ does not contain an infinite sharply 3-transitive subgroup. 
\end{abstract}
%\date{\today}

\end{minipage}

%Add references: new examples of non-split S2T groups : Rips-Tent, André-Tent, André-Guirardel, Amelio-André-Tent, de la Nuez-Sullivan

%Quelques idées / remarques : 

%1) GL(n,R) agit sur M(n,R) par conjugaison. Le noyau de cette action est le centred de GL(n,R) (qui est égal aux homothéties), donc PGL(n,R) agit fidèlement sur M(n,R), ce qui fournit un plongement de PGL(n,R) dans GL(n^2,R). Donc GL(9,R) contient un groupe S3T.

\section{Introduction}

\footnotetext[0\def\thefootnote{}]{The first author was supported until August 2025 by the German Research Foundation (DFG) under Germany’s Excellence Strategy EXC 2044–390685587, Mathematics Münster: Dynamics–Geometry–Structure and by CRC 1442 Geometry: Deformations and Rigidity. The first author was also supported since November 2026 by the German Research Foundation (DFG) under the grant 541703614 `Mapping class groups from above and below'.}

Let $n\geq 1$ be an integer. An action of a group $G$ on a set $X$ with $\vert X\vert \geq n$ is said to be sharply $n$-transitive if, for any two $n$-tuples of distinct elements in $X$, there exists a unique $g\in G$ mapping the first to the second. We say that a group $G$ is sharply $n$-transitive if there exists a set $X$ with a sharply $n$-transitive action of $G$ on $X$. Sharply $n$-transitive groups are fully classified for $n\geq 4$ and are known to be finite. In contrast, the cases $n=2$ and $n=3$ include infinite groups, such as $\mathrm{AGL}_1(K)$ acting on a field $K$ and $\mathrm{PGL}_2(K)$ for its natural action on $\mathbb{P}^1(K)$, respectively.

We define a sharply 2-transitive group as split if it contains a non-trivial normal abelian subgroup, or equivalently if it is isomorphic to $\mathrm{AGL}_1(K)$ for some near-field $K$ (see \cite{kerby} or \cite[Section 11.4]{BN94} for the definition of a near-field); similarly, a sharply 3-transitive group is split if the stabiliser of a point, which is a sharply 2-transitive group, is split. Whether every sharply 2-transitive group is necessarily split remained a long-standing open problem until Rips, Segev, and Tent constructed the first non-split examples in \cite{rips_segev_tent} (see also \cite{rips_tent,andre_tent,andre_guir_fin_gen_simple,amelio_andre_tent,amelio,nuez_sullivan} for further examples of non-split sharply 2-transitive groups). Tent later constructed the first example of a non-split sharply $3$-transitive group in \cite{tentS3T}, and de la Nuez González and Sullivan subsequently obtained new examples in \cite{nuez_sullivan} that are, remarkably, finitely presented. For more details on sharply 2- and 3-transitive groups, see \cite{kerby,BN94,tent,tent2}. Recently, Glasner and Gulko proved in \cite{glasnergulko} that $\mathrm{SL}_3(\mathbb{R})$ contains a non-split sharply 2-transitive group by adapting the construction of \cite{rips_segev_tent}, while proving that $\mathrm{SL}_2(K)$ contains no such subgroup for any field $K$. Then, they asked the following question:

\begin{qu}
Does the group $\mathrm{SL}_n(R)$ contain non-split sharply 2-transitive subgroups for smaller rings such as the field of algebraic numbers $R=\bar{\mathbb{Q}}$, the rationals $R=\mathbb{Q}$ or even the integers $R=\mathbb{Z}$? 
\end{qu}

In this short note, using results of \cite{tent_ziegler2}, we prove the following result, which answers this question positively for $R=\mathbb{Z}$.

\begin{te}\label{theorem1}
The group $\mathrm{SL}_3(\mathbb{Z})$ contains a non-split sharply 2-transitive subgroup.
\end{te}

\begin{rk}
It is worth noting that $\mathrm{SL}_3(\mathbb{Z})$ does not contain an infinite split sharply 2-transitive subgroup. Indeed, if $G$ is a split sharply 2-transitive group, then it is isomorphic to $N\rtimes H$ for some abelian group $N$, where $H$ acts on $N\setminus \lbrace 1\rbrace$ transitively. By a theorem due to Mal'cev \cite{Malcev}, every abelian subgroup of $\mathrm{GL}_n(\mathbb{Z})$ is finitely generated. Hence, if $G$ embeds into $\mathrm{GL}_n(\mathbb{Z})$ then $N$ is finitely generated, therefore $N$ must be finite, otherwise $\mathrm{Aut}(N)$ could not act transitively on $N\setminus \lbrace 1\rbrace$, which proves that $G$ must be finite.
\end{rk}

We also address the case of non-split sharply 3-transitive groups, using the main result of \cite{tentS3T}.

\begin{te}\label{theorem2}
The group $\mathrm{SL}_4(\mathbb{Z})$ contains a non-split sharply 3-transitive subgroup, but the group $\mathrm{SL}_3(\mathbb{Z})$ does not contain an infinite sharply 3-transitive subgroup, whether split or non-split. 
\end{te}

\begin{rk}
For the same reason as above, $\mathrm{SL}_4(\mathbb{Z})$ does not contain an infinite split sharply 3-transitive subgroup.
\end{rk}

\subsection*{Acknowledgments} The authors warmly thank the anonymous referee for their comments, which have helped to improve the clarity of the paper. We also thank Céline Berthaud for pointing out a mistake in Remark~4.4 of the first version of the paper.

\section{Projective space and the ping-pong lemma}

We first recall the ping-pong lemma for amalgamated products (see \cite[Chapter III, Proposition 12.4]{lyndonschupp}). In the next section, we only need the lemma for free products, but the version for amalgamated products will be needed in the final section.

\begin{lemme}\label{ping-pong}
Let $A,B$ be subgroups of a group $G$. Define
$C=A\cap B$ and suppose that $[A:C]>2$ and $[B:C]\geq 2$. Suppose that $G$ acts on a set $X$ containing two disjoint non-empty subsets $X_A,X_B$ such that, for all $a\in A\setminus C$, $b\in B\setminus C$ and $c\in C$, we have $a\cdot X_B\subset X_A$, $b\cdot X_A\subset X_B$, $c\cdot X_A\subset X_A$ and $c\cdot X_B\subset X_B$. Then the group $\langle A,B\rangle$ is isomorphic to $A\ast_C B$.\end{lemme}

In this paper, we will use the ping-pong lemma to construct amalgamated products in $\mathrm{SL}_n(\mathbb{Z})$ (for $n=3$ or $n=4$) using the natural action of $\mathrm{GL}_n(\mathbb{R})$ on the projective space $\mathrm{P}(\mathbb{R}^n)$. We will say that a matrix $A\in \mathrm{GL}_n(\mathbb{R})$ is \emph{proximal} if the following conditions hold: $A$ is diagonalisable, its eigenvalues belong to $\mathbb{R}$ and satisfy the following inequalities, possibly after renumbering: $\lambda_1 > \lambda_2\geq \cdots \geq \lambda_{n-1}\geq\lambda_n>0$ (hence the eigenspace corresponding to $\lambda_1$ is one-dimensional). Note that these conditions could be considerably relaxed, but that would not serve our purpose. Let $\pi : \mathbb{R}^n\rightarrow\mathrm{P}(\mathbb{R}^n)$ denote the natural surjection. Let $v_i$ denote an eigenvector associated with $\lambda_i$, and define $P_A^{+}=\pi(v_1)$ and $H_A^{+}=\pi(\langle v_2,\ldots,v_n\rangle)$. The point $P_A^{+}$ will be called the attracting point of $A$, and $H^{+}_A$ will be called the repelling hyperplane of $A$. This terminology is justified by the following observation: if $w\in\mathbb{R}^n=\langle v_1\rangle\oplus\langle v_2,\ldots,v_n\rangle$ does not belong to $\langle v_2,\ldots,v_n\rangle$, then it is not difficult to see that $A^n(w)/\norm{A^n(w)}$ converges to $\pm v_1/\norm{v_1}$ as $n$ tends to $+\infty$, and thus $A^n(\pi(w))$ converges to $P_A^{+}$. If $A^{-1}$ is also proximal (i.e.\ if the eigenvalues satisfy $\lambda_1 > \lambda_2\geq \cdots \geq \lambda_{n-1}>\lambda_n>0$), we denote by $P_{A}^{-}$ the attracting point of $A^{-1}$ and by $H^{-}_A$ the repelling hyperplane of $A^{-1}$, and we say that $A$ is \emph{biproximal}. Here is an important application: if two matrices $A,B\in\mathrm{GL}_n(\mathbb{R})$ are biproximal and satisfy $\lbrace P_A^{+},P_A^{-}\rbrace\cap (H_B^{+}\cup H_B^{-})=\varnothing$ and $\lbrace P_B^{+},P_B^{-}\rbrace\cap (H_A^{+}\cup H_A^{-})=\varnothing$, then the subgroup of $\mathrm{GL}_n(\mathbb{R})$ generated by $A^n$ and $B^n$ for some large integer $n$ is free. Indeed, with the notation of Lemma \ref{ping-pong}, one can take for $X_A$ (resp.\ $X_B$) the union of two closed balls of small radius centred at $P_A^{+}$ and $P_A^{-}$ (resp.\ $P_B^{+}$ and $P_B^{-}$).

%Let $V_i$ denote the subspace of $\mathbb{R}^n$ spanned by $\lbrace v_j \ \vert \ 1\leq j \leq n, \ j\neq i\rbrace$ and let $\pi : \mathbb{R}^n\rightarrow \mathrm{P}(\mathbb{R}^n)$ denote the natural projection. A proximal matrix $A$ enjoys the following nice property: if $B(P_1,\varepsilon)$ denotes the ball of small radius $\varepsilon>0$ centred at $P_1$ in $\mathrm{P}(\mathbb{R}^n)$ and if $K\subset \mathrm{P}(\mathbb{R}^n)$ is compact and satisfies $K\cap \pi(V_1)=\varnothing$, then $A^n(K)$ is contained in $B(P_1,\varepsilon)$ for any $n\in\mathbb{N}$ sufficiently large. Similarly, if $K\subset \mathrm{P}(\mathbb{R}^n)$ is compact and satisfies $K\cap \pi(V_n)=\varnothing$, then $A^{-n}(K)$ is contained in $B(P_n,\varepsilon)$ for any $n\in\mathbb{N}$ sufficiently large.

\section{Proof of Theorem \ref{theorem1}}

In \cite{tent_ziegler2}, the authors show that the group $(F_{\omega}\times \mathbb{Z}/2\mathbb{Z})\ast F_{\omega}$ is sharply 2-transitive. This group is a subgroup of $G=(F_2\times \mathbb{Z}/2\mathbb{Z})\ast F_2$, which is virtually free. Moreover, it is well known that $F_2$ embeds into $\mathrm{SL}_2(\mathbb{Z})$, therefore $G$ embeds into $\mathrm{SL}_n(\mathbb{Z})$ for some integer $n\geq 2$ (via the induced representation). Glasner and Gulko proved that $n$ must be equal to or greater than $3$. We will see that it is not difficult to construct an explicit embedding of $(F_2\times \mathbb{Z}/2\mathbb{Z})\ast F_2$ in $\mathrm{SL}_3(\mathbb{Z})$ using a ping-pong argument for the action of $\mathrm{SL}_3(\mathbb{Z})$ on the projective plane $\mathrm{P}(\mathbb{R}^3)$. 

\begin{te}\label{theorem1bis}
The group $\mathrm{SL}_3(\mathbb{Z})$ contains a non-split sharply 2-transitive subgroup. More precisely, $(F_2\times \mathbb{Z}/2\mathbb{Z})\ast F_2$ embeds into $\mathrm{SL}_3(\mathbb{Z})$
\end{te}

\begin{proof}
Consider the following four matrices in $\mathrm{SL}_3(\mathbb{Z})$, and note that $I$ commutes with $A$ and $B$:
\[I=\begin{pmatrix}
-1 & 0 & 0  \\
0 & -1 & 0  \\
0 & 0 & 1 
\end{pmatrix}, \ \ A=\begin{pmatrix}
2 & 1 & 0   \\
1 & 1 & 0  \\
0 & 0 & 1  
\end{pmatrix}, \ \ B=\begin{pmatrix}
3 & 2 & 0   \\
1 & 1 & 0  \\
0 & 0 & 1  
\end{pmatrix}, \ \ U=\begin{pmatrix}
0 & 0 & -1   \\
0 & 1 & 0  \\
1 & 0 & 0  
\end{pmatrix}.\]

\begin{comment}
, \ \ V=\begin{pmatrix}
1 & 0 & 0   \\
0 & 1 & 0  \\
1 & 0 & 1  
\end{pmatrix}  
\end{comment}

%The spectrum of $A$ is $\lbrace (3+\sqrt{5})/2, 1,(3-\sqrt{5})/2\rbrace$. 

The eigenvalues of $A$ are $\lambda_1=(3+\sqrt{5})/{2}>\lambda_2=1>\lambda_3=(3-\sqrt{5})/{2}>0$, and the corresponding eigenvectors are
\[
v_1=\begin{pmatrix}1+\sqrt5\\2\\0\end{pmatrix},\quad
v_2=\begin{pmatrix}1-\sqrt5\\2\\0\end{pmatrix},\quad
v_3=\begin{pmatrix}0\\0\\1\end{pmatrix}.
\]
In particular, since all eigenvalues of $A$ are positive and distinct, $A$ is biproximal. The repelling hyperplane $H^{+}_A$, which is spanned (projectively) by $v_2$ and $v_3$, is parametrised by $2x+(\sqrt{5}-1)y=0$ for $(x,y,z)\in\mathbb{R}^3$. Similarly, the repelling hyperplane $H_A^{-}$ is spanned by $v_1$ and $v_2$ and is parametrised by $2x-(\sqrt{5}+1)y=0$. Then, the eigenvalues of $B$ are $\mu_1=2+\sqrt3>\mu_2=1>\mu_3=2-\sqrt3>0$, and the corresponding eigenvectors are
\[
w_1=\begin{pmatrix}1+\sqrt3\\1\\0\end{pmatrix},\quad
w_2=\begin{pmatrix}1-\sqrt3\\1\\0\end{pmatrix},\quad
w_3=\begin{pmatrix}0\\0\\1\end{pmatrix}.
\]
Hence, $B$ is biproximal. A calculation shows that the repelling hyperplanes $H_B^{+},H_B^{-}$ are parametrised by the equations $x+(\sqrt{3}-1)y=0$ and $x-(\sqrt{3}+1)y=0$, respectively. 

Then, define $C=UBU^{-1}$. The attracting points of $C$ and $C^{-1}$ are respectively $P_C^{+}=U(P_B^{+})=[0:1:1+\sqrt{3}]$ and $P_C^{-}=U(P_B^{-})=[0:1:1-\sqrt{3}]$, and their repelling hyperplanes are $H_C^{+}=U(H_B^{+})$, which is parametrised by the equation $(\sqrt{3}-1)y+z=0$, and $H_C^{-}=U(H_B^{-})$, which is parametrised by the equation $(\sqrt{3}+1)y-z=0$.

%Define $C=UBU^{-1}$ and $D=VBV^{-1}$. The attracting points of $C$ and $C^{-1}$ are $P_C^{+}=U(P_B^{+})=[0:1:1+\sqrt{3}]$ and $P_C^{-}=U(P_B^{-})=[0:1:1-\sqrt{3}]$, and their repelling planes are $R_C^{+}=U(R_B^{+})$ parametrised by the equation $(\sqrt{3}-1)y+z=0$ and $R_C^{-}=U(R_B^{-})$ parametrised by the equation $(\sqrt{3}+1)y-z=0$. The attracting points of $D$ and $D^{-1}$ are $P_D^{+}=V(P_B^{+})=[1+\sqrt{3}:1:1+\sqrt{3}]$ and $P_D^{-}=V(P_B^{-})=[1-\sqrt{3}:1:1-\sqrt{3}]$, and their repelling planes are $R_D^{+}=V(R_B^{+})$ parametrised by the equation $x+(\sqrt{3}-1)y=0$ and $R_D^{-}=V(R_B^{-})$ parametrised by the equation $x-(\sqrt{3}+1)y=0$ (note in particular that $R_B^{+}=R_D^{+}$ and $R_B^{-}=R_D^{-}$).

Note that the points $P_A^{+},P_A^{-},P_B^{+},P_B^{-},P_C^{+},P_C^{-},I(P_C^{+}),I(P_C^{-})$ are pairwise distinct.
Note also that for any distinct $M,M'\in\lbrace A,B,C\rbrace$, the points $P_M^{+}$ and $P_M^{-}$ do not belong to $H_{M'}^{+}\cup H_{M'}^{-}$. For $M\in\lbrace A,B,C\rbrace$, define $N_M=\bar{B}(P_M^{+},\varepsilon)\cup \bar{B}(P_M^{-},\varepsilon)$ where $\bar{B}(P,\varepsilon)$ denotes the closed ball of radius $\varepsilon$ centred at $P$ in $\mathbb{P}^2(\mathbb{R})$. For $\varepsilon$ sufficiently small, $\bar{B}(P_M^{+},\varepsilon)$ and $\bar{B}(P_M^{-},\varepsilon)$ have empty intersection, and the following conditions hold:

\begin{itemize}
    \item $N_A\cap (N_B\cup N_C\cup H_B^{+}\cup H_B^{-}\cup H_C^{+}\cup H_C^{-})=\varnothing$,
    \item  $N_B\cap (N_A\cup N_C\cup H_A^{+}\cup H_A^{-}\cup H_C^{+}\cup H_C^{-})=\varnothing$,
    \item $N_C\cap (N_A\cup N_B\cup H_A^{+}\cup H_A^{-}\cup H_B^{+}\cup H_B^{-}\cup I(N_C))=\varnothing$.
\end{itemize}

Define $H=\langle A,B,I\rangle$ and $K=\langle C\rangle$. Then, define the sets $X_H=N_A\cup N_B\cup I(X_K)$ and $X_K=N_C$. Note that for any point $P\in\mathrm{P}(\mathbb{R}^3)$ that does not lie in $H_A^{+}\cup H_A^{-}$, the sequence $(A^nP)_{n\in\mathbb{N}}$ converges to $P_A^{+}$ as $n$ goes to $+\infty$ and converges to $P_A^{-}$ as $n$ goes to $-\infty$. The same observation holds for $B$ and $C$ as well. Therefore, by passing to sufficiently high powers of $A,B,C$, we may assume that the following conditions hold:
\begin{itemize}
    \item $A^{\pm 1}(X_H\cup X_K)\subset \bar{B}(P_A^{\pm 1},\varepsilon)\subset N_A\subset X_H$,
    \item $B^{\pm 1}(X_H\cup X_K)\subset \bar{B}(P_B^{\pm 1},\varepsilon)\subset N_B\subset X_H$,
    \item $C^{\pm1}(X_H\cup X_K)\subset \bar{B}(P_C^{\pm 1},\varepsilon) \subset N_C=X_K$.
\end{itemize}

We also have $I(X_K)\subset X_H$ by definition of $X_H$. Hence, by the ping-pong lemma \ref{ping-pong}, we have $H\simeq F_2\times \mathbb{Z}/2\mathbb{Z}$ and $\langle H,K\rangle\simeq H\ast K\simeq (F_2\times \mathbb{Z}/2\mathbb{Z})\ast \mathbb{Z}$. Finally, the subgroup of $\langle A,B,I,C\rangle$ generated by $H,ACA^{-1},BCB^{-1}$ is isomorphic to $(F_2\times \mathbb{Z}/2\mathbb{Z})\ast F_2$ (this can be proved by using normal forms in the free product $H\ast K$ or by applying the ping-pong lemma in the Bass-Serre tree of $H\ast K$).\end{proof}

\section{Proof of Theorem \ref{theorem2}}

In \cite{tentS3T}, Tent proved that the group $((F_{\omega}\times \mathbb{Z}/3\mathbb{Z}){}_{\mathbb{Z}/3\mathbb{Z}}\ast {S_3}\ast_{\mathbb{Z}/2\mathbb{Z}}( \mathbb{Z}/2\mathbb{Z}\times F_{\omega}))\ast F_{\omega}$ is a non-split sharply 3-transitive group (of characteristic 2). This group embeds into the group $G=((F_2\times \mathbb{Z}/3\mathbb{Z}){}_{\mathbb{Z}/3\mathbb{Z}}\ast {S_3}\ast_{\mathbb{Z}/2\mathbb{Z}}( \mathbb{Z}/2\mathbb{Z}\times F_2))\ast F_2$ (whose isomorphism class is independent of the choice of embeddings of $\mathbb{Z}/3\mathbb{Z}$ into $F_2\times \mathbb{Z}/3\mathbb{Z}$ and $S_3$, as well as the embeddings of $\mathbb{Z}/2\mathbb{Z}$ into $S_3$, since transpositions and 3-cycles are conjugate in $S_3$). Note that this group $G$ is the fundamental group of a finite graph of finite groups, hence it is virtually free, therefore it embeds into $\mathrm{SL}_n(\mathbb{Z})$ for some integer $n\geq 2$. We will prove that $G$ embeds into $\mathrm{SL}_4(\mathbb{Z})$, and that this result is optimal: the smallest integer $n$ such that $\mathrm{GL}_n(\mathbb{Z})$ contains an infinite sharply 3-transitive group $G$ (regardless of whether $G$ is split or not) is $n=4$, and in particular, $n=4$ is the smallest integer such that $\mathrm{SL}_n(\mathbb{Z})$ contains a non-split sharply 3-transitive group since all finite sharply 3-transitive groups are split.

\begin{te}
The group $\mathrm{GL}_3(\mathbb{Z})$ does not contain an infinite sharply 3-transitive subgroup (so, in particular, it does not contain a non-split sharply 3-transitive subgroup).
\end{te}

This is an immediate consequence of the following two lemmas.

\begin{lemme}
Let $G$ be an infinite group acting sharply 3-transitively on a set $X$. If $g\in G$ is an element of order 3, then the centraliser $C_G(g)$ is infinite.
\end{lemme}

\begin{proof}Let $F\subset X$ be the set of fixed points of $g$. Note that $\vert F\vert \leq 2$, so $X\setminus F$ is infinite. Fix a point $x\in X\setminus F$, and let $f:X\setminus F\rightarrow G$ be the function that maps every point $y\in X\setminus F$ to the unique element $h_y$ of $G$ such that $(x,g(x),g^2(x))=h_y(y,g(y),g^2(y))$. Note that $f$ is injective. Then, observe that $h_y^{-1}gh_y(y)=h_y^{-1}g(x)=g(y)$. Similarly, $h_y^{-1}gh_y$ and $g$ coincide on $g(y)$ and on $g^2(y)$. It follows that $h_y^{-1}gh_y=g$ and thus $h_y$ belongs to $C_G(g)$. Therefore, $C_G(g)$ is infinite.
\end{proof}

\begin{lemme}
The centraliser of any element of order 3 in $\mathrm{GL}_3(\mathbb{Z})$ is finite.
\end{lemme}

\begin{proof}Let $M\in \mathrm{GL}_3(\mathbb{Z})$ be an element of order 3. It is not difficult to prove that $M$ is conjugate to one of the following two matrices (see for instance page 9 of \cite{AC91}):
\[M_1=\begin{pmatrix}
1 & 0 & 0  \\
0 & 0 & -1 \\
0 & 1 & -1
\end{pmatrix} \ \ \text{ and } \ \ M_2=\begin{pmatrix}
1 & 0 & 1  \\
0 & 0 & -1 \\
0 & 1 & -1
\end{pmatrix}\]
Finally, a calculation shows that the centraliser of $M_1$ in $\mathrm{GL}_3(\mathbb{Z})$ is of order 12 and that the centraliser of $M_2$ in $\mathrm{GL}_3(\mathbb{Z})$ is of order 6.
\end{proof}

%First, note that $G$ does not embed into $\mathrm{SL}_3(\mathbb{Z})$ because the centraliser of every element of order 3 in $\mathrm{SL}_3(\mathbb{Z})$ is finite. We will prove the following result.

\begin{rk}If $G$ is sharply 3-transitive and $g\in G$ is of order 3, then $g$ is conjugate to $g^{-1}$ (by an involution). Therefore, if $G$ embeds into $\mathrm{GL}_3(\mathbb{C})$, then $g$ must have three distinct eigenvalues and thus the centraliser of $g$ in $\mathrm{GL}_3(\mathbb{C})$ is abelian. It follows that the non-split sharply 3-transitive group constructed by Tent in \cite{tentS3T} does not embed into $\mathrm{GL}_3(\mathbb{C})$. However, we do not know the answer to the following questions (a positive answer to the second question would give a positive answer to the first question):
\begin{enumerate}
    \item Is there a non-split sharply 3-transitive group such that the centraliser of any element of order 3 is abelian?
    \item Does $\mathrm{GL}_3(\mathbb{C})$ contain a non-split sharply 3-transitive subgroup?
\end{enumerate}
\end{rk}

\begin{te}The group $\mathrm{SL}_4(\mathbb{Z})$ contains a non-split sharply 3-transitive subgroup. More precisely, the group $((F_2\times \mathbb{Z}/3\mathbb{Z}){}_{\mathbb{Z}/3\mathbb{Z}}\ast {S_3}\ast_{\mathbb{Z}/2\mathbb{Z}}( \mathbb{Z}/2\mathbb{Z}\times F_2))\ast F_2$ embeds into $\mathrm{SL}_4(\mathbb{Z})$.
\end{te}

\begin{proof}
Consider the following four matrices in $\mathrm{SL}_4(\mathbb{Z})$: 
\[I=\begin{pmatrix}
0 & 1 & 0 & 0  \\
1 & 0 & 0 & 0 \\
0 & 0 & -1 & 0 \\
0 & 0 & 0 & 1 
\end{pmatrix}, \ J=\begin{pmatrix}
0 & -1 & 0 & 0  \\
1 & -1 & 0 & 0 \\
0 & 0 & 1 & 0 \\
0 & 0 & 0 & 1 
\end{pmatrix}, \ A=\begin{pmatrix}
1 & 0 & 0 & 0  \\
0 & 1 & 0 & 0 \\
0 & 0 & 2 & 1 \\
0 & 0 & 1 & 1 
\end{pmatrix}, \ B=\begin{pmatrix}
1 & 0 & 0 & 0  \\
0 & 1 & 0 & 0 \\
0 & 0 & 3 & 1 \\
0 & 0 & 2 & 1 
\end{pmatrix}.\]

Note that $I$ is of order 2, that $J$ is of order 3, that $\langle I\rangle$ normalises $\langle J\rangle$ but that $I$ and $J$ do not commute, so $\langle I,J\rangle$ is isomorphic to the symmetric group $S_3$. Note also that $A,B$ commute with $J$. We will prove that $\langle I,A,B\rangle\simeq \mathbb{Z}/2\mathbb{Z}\ast F_2$. It follows that $\langle I,J,A,B\rangle$ is isomorphic to $\mathbb{Z}/3\mathbb{Z}\rtimes (\mathbb{Z}/2\mathbb{Z}\ast F_2)$ (where the $\mathbb Z / 2 \mathbb Z$ factor inverts the generators of $\mathbb Z / 3 \mathbb Z$, and the $F_2$ factor fixes them), which is isomorphic to $(F_2\times \mathbb{Z}/3\mathbb{Z})\ast_{\mathbb{Z}/3\mathbb{Z}}S_3$.

The eigenvalues of $A$ are $\alpha_1=(3+\sqrt{5})/2>\alpha_2=\alpha_3=1>\alpha_4=(3-\sqrt{5})/2>0$, so $A$ is biproximal. The corresponding eigenvectors are 
\[
a_1=\begin{pmatrix}0\\0\\1+\sqrt5\\2\end{pmatrix},\quad
a_2=\begin{pmatrix}0\\0\\1-\sqrt5\\2\end{pmatrix},\quad
a_3=\begin{pmatrix}1\\0\\0\\0\end{pmatrix},\quad
a_4=\begin{pmatrix}0\\1\\0\\0\end{pmatrix}.
\] 
The repelling hyperplanes are \[H_A^{+}=\lbrace [x_1:x_2:x_3:x_4]\in\mathrm{P}(\mathbb{R}^4) \ \vert \ 2x_3-(1-\sqrt{5})x_4=0\rbrace,\]
\[H_A^{-}=\lbrace [x_1:x_2:x_3:x_4]\in\mathrm{P}(\mathbb{R}^4) \ \vert \ 2x_3-(1+\sqrt{5})x_4=0\rbrace.\]

The eigenvalues of $B$ are $\beta_1=2+\sqrt{3}>\beta_2=\beta_3=1>\beta_4=2-\sqrt{3}>0$, so $B$ is biproximal. The corresponding eigenvectors are \[
b_1=\begin{pmatrix}0\\0\\1+\sqrt3\\2\end{pmatrix},\quad
b_2=\begin{pmatrix}0\\0\\1-\sqrt3\\2\end{pmatrix},\quad
b_3=\begin{pmatrix}1\\0\\0\\0\end{pmatrix},\quad b_4=\begin{pmatrix}0\\1\\0\\0\end{pmatrix}.
\]
The repelling hyperplanes are \[H_B^{+}=\lbrace [x_1:x_2:x_3:x_4]\in\mathrm{P}(\mathbb{R}^4) \ \vert \ 2x_3-(1-\sqrt{3})x_4=0\rbrace,\]\[H_B^{-}=\lbrace [x_1:x_2:x_3:x_4]\in\mathrm{P}(\mathbb{R}^4) \ \vert \ 2x_3-(1+\sqrt{3})x_4=0\rbrace.\]

Note that $P_A^{+}$ and $P_A^{-}$ do not belong to $H_B^{+}\cup H_B^{-}$ and that $P_B^{+}$ and $P_B^{-}$ do not belong to $H_A^{+}\cup H_A^{-}$. For $M\in\lbrace A,B\rbrace$, define $N_M=\bar{B}(P_M^{+},\varepsilon)\cup \bar{B}(P_M^{-},\varepsilon)$ where $\bar{B}(P,\varepsilon)$ denotes the closed ball of radius $\varepsilon$ centred at $P$ in $\mathrm{P}(\mathbb{R}^4)$. For $\varepsilon$ sufficiently small, $\bar{B}(P_M^{+},\varepsilon)$ and $\bar{B}(P_M^{-},\varepsilon)$ have an empty intersection, and we have $N_A\cap (N_B\cup H_B^{+}\cup H_B^{-})=\varnothing$ and $N_B\cap (N_A\cup H_A^{+}\cup H_A^{-})=\varnothing$. Taking sufficiently large powers of $A$ and $B$, we can therefore assume that $A^{\pm 1}(N_B)\subset \bar{B}(P_A^{\pm 1},\varepsilon)\subset N_A$ and $B^{\pm 1}(N_A)\subset \bar{B}(P_B^{\pm 1},\varepsilon)\subset N_B$, so the group $\langle A,B\rangle$ is isomorphic to the free group $F_2$ by the ping-pong lemma \ref{ping-pong}.

Define the ping-pong sets $X=N_A\cup N_B$ for $\langle A,B\rangle$ and $Y=I(X)$ for $\langle I\rangle$. Notice that the points $I(P_A^{\pm} )$ and $I(P_B^{\pm})$ are pairwise distinct from the points $P_A^{\pm 1}$ and $P_B^{\pm 1}$, so the sets $X$ and $Y$ are disjoint, provided the radius $\varepsilon$ of the balls is small enough. Note that $A^{\pm 1}(X\cup Y)\subset \bar{B}(P_A^{\pm 1},\varepsilon)\subset N_A\subset X$ and $B^{\pm 1}(X\cup Y)\subset\bar{B}(P_B^{\pm 1},\varepsilon)\subset N_B\subset X$ (again, after replacing $A$ and $B$ by $A^k$ and $B^k$ for $k$ sufficiently large), and that $I(X)\subset Y$ by definition of $Y$, so $\langle I,A,B\rangle$ is isomorphic to $ \mathbb{Z}/2\mathbb{Z}\ast F_2$ by the ping-pong lemma \ref{ping-pong}.

Next, we examine the following two matrices $C,D\in \mathrm{SL}_4(\mathbb{Z})$, which belong to the centraliser of the involution $I$:
\[C=\begin{pmatrix}
2 & -1 & 1 & 0  \\
-1 & 2 & -1 & 0 \\
1 & -1 & 1 & 0 \\
0 & 0 & 0 & 1 
\end{pmatrix}, \ \ D=\begin{pmatrix}
4 & -1 & 1 & 1  \\
-1 & 4 & -1 & 1 \\
2 & -2 & 1 & 0 \\
1 & 1 & 0 & 1
\end{pmatrix}.\]

The eigenvalues of $C$ are $\gamma_1=2+\sqrt{3}>\gamma_2=\gamma_3=1>\gamma_4=2-\sqrt{3}>0$, so $C$ is biproximal. The corresponding eigenvectors are
\[c_1=\begin{pmatrix}1\\-1\\\sqrt3-1\\0\end{pmatrix},\quad
c_2=\begin{pmatrix}1\\1\\0\\0\end{pmatrix},\quad c_3=\begin{pmatrix}0\\0\\0\\1\end{pmatrix},\quad c_4=\begin{pmatrix}1\\-1\\-(\sqrt3+1)\\0\end{pmatrix}.
\]
The repelling hyperplanes are \[H_C^{+}=\lbrace [x_1:x_2:x_3:x_4]\in\mathrm{P}(\mathbb{R}^4) \ \vert \ (\sqrt{3}+1)(x_1-x_2)+2x_3=0\rbrace,\]\[H_C^{-}=\lbrace [x_1:x_2:x_3:x_4]\in\mathrm{P}(\mathbb{R}^4) \ \vert \ (1-\sqrt{3})(x_1-x_2)+2x_3=0\rbrace.\] 

The eigenvalues of $D$ are $\delta_1=3+2\sqrt{2}>\delta_2=2+\sqrt{3}>\delta_3=2-\sqrt{3}>\delta_4=3-2\sqrt{2}>0$, so $D$ is biproximal. The corresponding eigenvectors are
\[
d_1=\begin{pmatrix}1+\sqrt2\\-(1+\sqrt2)\\2\\0\end{pmatrix},\quad
d_2=\begin{pmatrix}1+\sqrt3\\1+\sqrt3\\0\\2\end{pmatrix},\quad
d_3=\begin{pmatrix}1-\sqrt3\\1-\sqrt3\\0\\2\end{pmatrix},\quad
d_4=\begin{pmatrix}1-\sqrt2\\-(1-\sqrt2)\\2\\0\end{pmatrix}.
\]

The repelling hyperplanes are \[H_D^{+}=\lbrace [x_1:x_2:x_3:x_4]\in\mathrm{P}(\mathbb{R}^4) \ \vert \ x_1-x_2+(\sqrt{2}-1)x_3=0\rbrace,\]\[H_D^{-}=\lbrace [x_1:x_2:x_3:x_4]\in\mathrm{P}(\mathbb{R}^4) \ \vert \ x_1-x_2-(1+\sqrt{2})x_3=0\rbrace.\]

%Note that $P_C$ does not belong to $R_D$ and that $P_D$ does not belong to $R_C$. Let $N_C,N_D$ be two closed balls of sufficiently small radius centred at $P_C,P_D$ respectively in $\mathrm{P}(\mathbb{R}^4)$ such that $N_C\cap (N_D\cup R_D)=\varnothing$ and $N_D\cap (N_C\cup R_C)=\varnothing$. After replacing $C$ and $D$ by $C^k$ and $D^k$ for some large $k$ if necessary, we can therefore assume that $C(N_D)\subset N_C$ and $D(N_C)\subset N_D$, so $\langle C,D\rangle\simeq F_2$ by the ping-pong lemma \ref{ping-pong}.

Note that $P_C^{+}$ and $P_C^{-}$ do not belong to $H_D^{+}\cup H_D^{-}$ and that $P_D^{+}$ and $P_D^{-}$ do not belong to $H_C^{+}\cup H_C^{-}$. For $M\in\lbrace C,D\rbrace$, define $N_M=\bar{B}(P_M^{+},\varepsilon)\cup \bar{B}(P_M^{-},\varepsilon)$ where $\bar{B}(P,\varepsilon)$ denotes the closed ball of radius $\varepsilon$ centred at $P$ in $\mathrm{P}(\mathbb{R}^4)$. For $\varepsilon$ sufficiently small, $\bar{B}(P_M^{+},\varepsilon)$ and $\bar{B}(P_M^{-},\varepsilon)$ have empty intersection, and we have $N_C\cap (N_D\cup H_D^{+}\cup H_D^{-})=\varnothing$ and $N_D\cap (N_C\cup H_C^{+}\cup H_C^{-})=\varnothing$. By taking sufficiently large powers of $C$ and $D$, we can therefore assume that $C^{\pm 1}(N_D)\subset \bar{B}(P_C^{\pm 1},\varepsilon)\subset N_C$ and $D^{\pm 1}(N_C)\subset \bar{B}(P_D^{\pm 1},\varepsilon)\subset N_D$, so the group $\langle C,D\rangle$ is isomorphic to the free group $F_2$ by the ping-pong lemma \ref{ping-pong}.

Define $H=\langle I,J,A,B\rangle$ and $K=\langle I,C,D\rangle$. We proved that $H\simeq (F_2\times \mathbb{Z}/3\mathbb{Z})\ast_{\mathbb{Z}/3\mathbb{Z}}S_3$ and that $K\simeq \mathbb{Z}/2\mathbb{Z}\times F_2$. We will prove that the group $\langle H,K\rangle$ is isomorphic to $H\ast_{\mathbb{Z}/2\mathbb{Z}} K$ with $\mathbb{Z}/2\mathbb{Z}$ identified with $\langle I\rangle$ in $H$ and $K$. Define $X_K=N_C\cup N_D\cup I(N_C)\cup I(N_D)$ and $X_H=N_A\cup N_B\cup I(N_A)\cup I(N_B)\cup J(X_K)\cup J^2(X_K)$. Note that $I(X_K)=X_K$ and $I(X_H)=X_H$ by definition of $X_K$ and $X_H$ and by the fact that $IJ=J^2I$ and $IJ^2=JI$.

A thorough verification shows that the set $\lbrace P^{\pm 1}_C,P^{\pm 1}_D\rbrace \cup H^{\pm 1}_C\cup H^{\pm 1}_D$
has an empty intersection with the following set:
\[
\begin{aligned}
\lbrace 
& P^{\pm 1}_A, P^{\pm 1}_B, I(P^{\pm 1}_A), I(P^{\pm 1}_B), J(P^{\pm 1}_C), J(P^{\pm 1}_D), \\
& JI(P^{\pm 1}_C), JI(P^{\pm 1}_D), J^2(P^{\pm 1}_C), J^2(P^{\pm 1}_D), J^2I(P^{\pm 1}_C), J^2I(P^{\pm 1}_D)
\rbrace.
\end{aligned}
\]
Thus, by taking large powers of $C$ and $D$, we may assume that $C^{\pm 1}(X_H\cup X_K)\subset N_C\subset X_K$ and $D^{\pm 1}(X_H\cup X_K)\subset N_D\subset X_K$, so for every $g\in K\setminus \langle I\rangle$ we have $g(X_H)\subset X_K$.

Then, one can verify that $\lbrace P^{\pm 1}_A,P^{\pm 1}_B\rbrace \cup H^{\pm 1}_A\cup H^{\pm 1}_B$ has an empty intersection with the set $\lbrace P^{\pm 1}_C,P^{\pm 1}_D,I(P^{\pm 1}_C),I(P^{\pm 1}_D)\rbrace$, so (after taking powers of $A$ and $B$ if necessary) we have $A^{\pm 1}(X_K)\subset N_A\subset X_H$ and $B^{\pm 1}(X_K)\subset N_B\subset X_H$, and moreover $A^{\pm 1}(X_H)\subset N_A\subset X_H$ and $B^{\pm 1}(X_H)\subset N_B\subset X_H$. Note also that $J(X_K)\subset X_H$ and $J^2(X_K)\subset X_H$ by definition of $X_H$.
Let $g\in H\setminus \langle I\rangle$. We can write $g=g'J^\delta$ with $\delta\in\lbrace 0,1,2\rbrace$ and $g'$ a reduced word in $I,A^{\pm 1},B^{\pm 1}$. If $\delta\neq 0$ then $g(X_K)\subset g'(X_H)\subset X_H$ since $J$ maps $X_K$ into $X_H$ and $I,A^{\pm 1},B^{\pm 1}$ map $X_H$ into $X_H$. Then, suppose that $\delta=0$. Write $g=g''\ell$ with $\ell\in\lbrace I,A^{\pm 1},B^{\pm 1}\rbrace$ and $g''$ a reduced word whose last letter is not $\ell^{-1}$. If $\ell\neq I$, we have $g(X_K)\subset g''(X_H)\subset X_H$. If $\ell=I$ then $g(X_K)\subset g''(X_K)$ (because $I(X_K)=X_K$), but $g\neq I$ so we can write $g=g'''\ell'\ell$ with $\ell'\in \lbrace A^{\pm 1},B^{\pm 1}\rbrace$ and we conclude in the same way. Hence, for every $g\in H\setminus \langle I\rangle$ we have $g(X_K)\subset X_H$. 

Conclusion: the group $\langle H,K\rangle$ is isomorphic to $H\ast_{\mathbb{Z}/2\mathbb{Z}} K$, and thus to \[(F_2\times \mathbb{Z}/3\mathbb{Z}){}_{\mathbb{Z}/3\mathbb{Z}}\ast {S_3}\ast_{\mathbb{Z}/2\mathbb{Z}}( \mathbb{Z}/2\mathbb{Z}\times F_2).\] Finally, it is an exercise to find two matrices $E,F$ such that $E^{\pm 1}$ and $F^{\pm 1}$ are proximal and verify $(P^{\pm 1}_E\cup H^{\pm 1}_E)\cap (P^{\pm 1}_F\cup H^{\pm 1}_F)=\varnothing$ and $(P^{\pm 1}_E\cup H^{\pm 1}_E)\cap (X^{\pm 1}_H\cup X^{\pm 1}_K)=\varnothing$ and $(P^{\pm 1}_F\cup H^{\pm 1}_F)\cap (X_H\cup X_K)=\varnothing$. Therefore, the group $\langle H,K,E,F\rangle$ is isomorphic to $G$ by the ping-pong lemma.\end{proof}

%\renewcommand{\refname}{Bibliography}
%\bibliographystyle{plain}
%\bibliography{mybibliography}
\printbibliography

@article{Malcev,
 author = {Mal'sev, A. I.},
 title = {On certain classes of infinite solvable groups},
 fjournal = {Translations. Series 2. American Mathematical Society},
 journal = {Transl., Ser. 2, Am. Math. Soc.},
 issn = {0065-9290},
 volume = {2},
 pages = {1--21},
 year = {1956},
 language = {English},
 doi = {10.1090/trans2/002/01},
 zbMATH = {3117586},
 Zbl = {0070.25404}
}

@article{AC91,
 author = {Auslander, L. and Cook, M.},
 title = {An algebraic classification of the three-dimensional crystallographic groups},
 fjournal = {Advances in Applied Mathematics},
 journal = {Adv. Appl. Math.},
 issn = {0196-8858},
 volume = {12},
 number = {1},
 pages = {1--21},
 year = {1991},
 language = {English},
 doi = {10.1016/0196-8858(91)90002-Z},
 zbMATH = {4208382},
 Zbl = {0731.20030}
}

@book{BN94,
 author = {Borovik, A. and Nesin, A.},
 title = {Groups of finite {Morley} rank},
 fseries = {Oxford Logic Guides},
 series = {Oxf. Logic Guides},
 volume = {26},
 isbn = {0-19-853445-0},
 year = {1994},
 publisher = {Oxford: Clarendon Press},
 language = {English},
 zbMATH = {705165},
 Zbl = {0816.20001}
}

@misc{nuez_sullivan,
 author = {de la Nuez Gonz{\'a}lez, J. and Sullivan, R.},
 title = {Sharply k-transitive actions on ultrahomogeneous structures},
 year = {2025},
 howpublished = {Preprint, {arXiv}:2502.11166 [math.{GR}] (2025)},
 url = {https://arxiv.org/abs/2502.11166},
 arXiv = {arXiv:2502.11166}
}

@article{glasnergulko,
 author = {Glasner, Y. and Gulko, D.},
 title = {Non-split linear sharply 2-transitive groups},
 fjournal = {Proceedings of the American Mathematical Society},
 journal = {Proc. Am. Math. Soc.},
 issn = {0002-9939},
 volume = {149},
 number = {6},
 pages = {2305--2317},
 year = {2021},
 language = {English},
 doi = {10.1090/proc/15360},
 zbMATH = {7337047},
 Zbl = {1535.20010}
}

@misc{amelio,
 author = {Amelio, M.},
 title = {Non-split sharply 2-transitive groups of bounded exponent},
 year = {2025},
 howpublished = {Preprint, {arXiv}:2509.11958 [math.{GR}] (2025)},
 url = {https://arxiv.org/abs/2509.11958},
 arXiv = {arXiv:2509.11958}
}

@book{lyndonschupp,
 author = {Lyndon, R. and Schupp, P.},
 title = {Combinatorial group theory.},
 edition = {Reprint of the 1977 ed.},
 fseries = {Classics in Mathematics},
 series = {Class. Math.},
 issn = {1431-0821},
 isbn = {3-540-41158-5},
 year = {2001},
 publisher = {Berlin: Springer},
 language = {English},
 zbMATH = {1554175},
 Zbl = {0997.20037}
}

@article{tentS3T,
 author = {Tent, K.},
 title = {Sharply 3-transitive groups.},
 fjournal = {Advances in Mathematics},
 journal = {Adv. Math.},
 issn = {0001-8708},
 volume = {286},
 pages = {722--728},
 year = {2016},
 language = {English},
 doi = {10.1016/j.aim.2015.09.018},
 zbMATH = {6506330},
 Zbl = {1336.20002}
}

@article{amelio_andre_tent,
 author = {Amelio, M. and Andr{\'e}, S. and Tent, K.},
 title = {Non-split sharply {{\(2\)}}-transitive groups of odd positive characteristic},
 fjournal = {IMRN. International Mathematics Research Notices},
 journal = {Int. Math. Res. Not.},
 issn = {1073-7928},
 volume = {2025},
 number = {19},
 pages = {33},
 note = {Id/No rnaf294},
 year = {2025},
 language = {English},
 doi = {10.1093/imrn/rnaf294},
 zbMATH = {8104101}
}

@article{andre_guir_fin_gen_simple,
  title={Finitely generated simple sharply 2-transitive groups},
  author={Andr{\'e}, S. and Guirardel, V.},
  journal={Compositio Mathematica},
  volume={160},
  number={8},
  pages={1941--1957},
  year={2024},
  publisher={London Mathematical Society}
}

@article {andre_tent,
    AUTHOR = {Andr\'{e}, S. and Tent, K.},
     TITLE = {Simple sharply 2-transitive groups},
   JOURNAL = {Trans. Amer. Math. Soc.},
  FJOURNAL = {Transactions of the American Mathematical Society},
    VOLUME = {376},
      YEAR = {2023},
    NUMBER = {6},
     PAGES = {3965--3993},
      ISSN = {0002-9947,1088-6850},
   MRCLASS = {20B22 (20E06 20E32)},
  MRNUMBER = {4586803},
       DOI = {10.1090/tran/8846},
       URL = {https://doi.org/10.1090/tran/8846},
}

@book{kerby,
  title={On infinite sharply multiply transitive groups},
  author={Kerby, W.},
  number={6},
  year={1974},
  publisher={Vandenhoeck \& Ruprecht}
}

@article{rips_segev_tent,
  title={A sharply 2-transitive group without a non-trivial abelian normal subgroup},
  author={Rips, E. and Segev, Y. and Tent, K.},
  journal={J. Eur. Math. Soc.(JEMS)},
  volume={19},
  number={10},
  pages={2895--2910},
  year={2017}
}

@article{rips_tent,
  title={Sharply 2-transitive groups of characteristic 0},
  author={Rips, E. and Tent, K.},
  journal={Journal f{\"u}r die reine und angewandte Mathematik (Crelles Journal)},
  volume={2019},
  number={750},
  pages={227--238},
  year={2019},
  publisher={De Gruyter}
}

@article{tent2,
  title={From the Cherlin-Zilber Conjecture via sharply 2-transitive groups to the
Burnside problem},
  author={Tent, K.},
  journal={Proc. of the ICM 2026 (to appear)},
  year={2025}
}

@article{tent,
  title={Infinite sharply multiply transitive groups},
  author={Tent, K.},
  journal={Jahresbericht der Deutschen Mathematiker-Vereinigung},
  volume={118},
  number={2},
  pages={75--85},
  year={2016},
  publisher={Springer}
}

@article{tent_ziegler2,
 author = {Tent, K. and Ziegler, M.},
 title = {Sharply 2-transitive groups.},
 fjournal = {Advances in Geometry},
 journal = {Adv. Geom.},
 issn = {1615-715X},
 volume = {16},
 number = {1},
 pages = {131--134},
 year = {2016},
 language = {English},
 doi = {10.1515/advgeom-2015-0047},
 zbMATH = {6535875},
 Zbl = {1343.20002}
}

\vspace{1cm}

\setlength{\parindent}{0pt}
Marco Amelio
\\
Institute for Algebra and Geometry at the Karlsruhe Insitut für Technologie (KIT)
\\
Englerstra{\ss}e 2, 76131 Karlsruhe, Germany.
\\
Email address: \href{mailto:marco.amelio@kit.edu}{marco.amelio@kit.edu}

\bigskip

Simon Andr{\' e}
\\
Sorbonne Universit{\' e} et Universit{\' e} Paris Cit{\' e}
\\
CNRS, Institut de mathématiques de Jussieu - Paris Rive Gauche
\\
75005 Paris, France.
\\
Email address: \href{mailto:simon.andre@imj-prg.fr}{simon.andre@imj-prg.fr}

\end{document}